\documentclass[12pt,a4paper]{amsart}
\usepackage{amsfonts, amsmath, amssymb, amscd, amsthm, array}

\usepackage{multicol}
\usepackage{subfig}
\usepackage{color}
\usepackage[all]{xy}
\usepackage{hyperref,url}

\makeatletter
\def\blfootnote{\xdef\@thefnmark{}\@footnotetext}
\makeatother

\newtheorem{thm}{Theorem}[section]
\newtheorem{cor}[thm]{Corollary}
\newtheorem{lemma}[thm]{Lemma}
\newtheorem{prop}[thm]{Proposition}

\theoremstyle{definition}
\newtheorem{qu}[thm]{Question}
\newtheorem{df}[thm]{Definition}

\theoremstyle{remark}
\newtheorem{rem}[thm]{Remark}
\newtheorem{ex}[thm]{Example}

\newcommand{\cA}{\mathcal{A}}
\newcommand{\cC}{\mathcal{C}}

\newcommand{\cV}{\mathcal{V}}

\newcommand{\D}{\mathbb{D}}
\newcommand{\F}{\mathbb{F}}
\newcommand{\N}{\mathbb{N}}
\newcommand{\Z}{{\mathbb Z}}
\newcommand{\R}{\mathbb{R}}
\newcommand{\C}{\mathbb{C}}

\newcommand{\G}{\mathfrak{G}}
\newcommand{\ga}{\Gamma}
\newcommand{\GG}{\Gamma\mathfrak{G}}

\DeclareMathOperator{\link}{link}
\DeclareMathOperator{\st}{star}

\DeclareMathOperator{\adim}{adim} 

\def\coloneqq{\mathrel{\mathop\mathchar"303A}\mkern-1.2mu=}
\newcommand{\gen}[1]{\left\langle#1\right\rangle}

\begin{document}

\title[The Haagerup property is stable under graph products]{The Haagerup property is stable under graph products}
\address{University of Neuch\^{a}tel,
Institut de Math\'{e}matiques, Rue Emile-Argand 11,CH-2000 Neuch\^{a}tel, Switzerland.}
\author{Yago Antol\'{i}n}
\email[Yago Antol\'{i}n]{yago.anpi@gmail.com}

\address{School of Mathematics,
University of Southampton, Highfield, Southampton, SO17 1BJ, United Kingdom.}
\author{Dennis Dreesen}
\email[Dennis Dreesen]{Dennis.Dreesen@soton.ac.uk}


\date{\footnotesize\today}

\begin{abstract} 
The Haagerup property, which is a strong converse of Kazhdan's property $(T)$, has translations and applications in various fields of mathematics such as representation theory, harmonic analysis, operator K-theory and so on. Moreover, this group property implies the Baum-Connes conjecture and related Novikov conjecture. The Haagerup property is not preserved under arbitrary group extensions and amalgamated free products over infinite groups, but it is preserved under wreath products and amalgamated free products over finite groups. In this paper, we show that it is also preserved under graph products. We moreover give bounds on the equivariant and non-equivariant $L_p$-compressions of a graph product in terms of the corresponding compressions of the vertex groups. Finally, we give an upper bound on the asymptotic dimension in terms of the asymptotic dimensions of the vertex groups. This generalizes a result from Dranishnikov on the asymptotic dimension of right-angled Coxeter groups.
\end{abstract}


\maketitle

\setcounter{tocdepth}{1}



\section{Introduction}
A graph product is a natural group-theoretic construction generalizing 
free products and direct product. 
The main examples of graph products are right-angled Artin and Coxeter groups.

Many properties that are stable under free and direct products 
are also stable under graph products. 
Some examples of these properties are 
left-orderability and bi-orderability \cite{Chiswell}, 
soficity \cite{CHR1}, rapid decay \cite{CHR2}, 
residual finiteness \cite{Green}, linearity \cite{HsuWise},  
semihyperbolicity and automaticity \cite{HermillerMeier}
or the Tits alternative \cite{AntolinMinasyan}, just to name a few.

The motivation of this paper is to answer the following question. 
\begin{qu}
Is the Haagerup property stable under graph products?
\end{qu}

We answer this question affirmatively, even when we consider actions on $L_p$-spaces instead of restricting to actions on Hilbert spaces.
Our answer is also quantitative, 
that is we measure ``how Haagerup" the graph product is,
and ``how distorted" a coarse embedding of a graph product into an $L_p$-space needs to be. This is done by 
studying  the compression functions 
of equivariant and non-equivariant embeddings of graph products 
in $L_p$-spaces. 

Finally, we also study the behavior of the asymptotic dimension under
graph products of groups, generalizing a result from Dranishnikov \cite{Dranishnikov}.
This is related to coarse embeddability  since 
groups of finite asymptotic dimension coarsely embed in Hilbert spaces \cite{Roe}.

\subsection{Coarse embeddability, the Haagerup property and compression}
\begin{df}[see \cite{Gromov91}]
Fix $p\geq 1$. A finitely generated group $G$ is {\it coarsely embeddable into an $L_p$-space}, if there exists a measure space $(\Omega,\mu)$, a non-decreasing function $\rho_-:\R^+ \rightarrow \R^+$ such that $\lim_{t\to \infty} \rho_-(t)=+\infty$, a constant $C>0$ and a map $f\colon G\rightarrow L_p(\Omega,\mu)$, such that
\[ \rho_-(d(g,h)) \leq \|f(g)-f(h)\|_p \leq C d(g,h) \ \forall g,h\in G, \]
where $d$ is the word length metric relative to a finite generating subset.
The map $f$ is called a {\it coarse embedding} of $G$ into $L_p(\Omega,\mu)$ and the map $\rho_-$ is called a {\it compression function} for $f$.
\end{df}
Coarse embeddability into the Hilbert space is related to deep conjectures: for example, it is known to imply the Novikov Conjecture \cite{Tu}. Later, in \cite{Yu:3}, the authors prove the same result for embeddings into uniformly convex Banach spaces and this is one of the motivations to study embeddings into $L_p$-spaces for $p\neq 2$.
\begin{df}
Let $p\geq 1$, let $G$ be a topological group and $(\Omega,\mu)$ a measure space. Given an affine isometric action of $G$ on $L_p(\Omega,\mu)$, let $b\colon G\rightarrow L_p(\Omega,\mu), g\mapsto g\cdot 0$, be the orbit map of $0$. We say that $b$ is proper if for every $M\geq 0$, there exists $K\subset G$ compact such that $\|b(g)\|\geq M$ whenever $g\in G\backslash K$.
 A locally compact second countable group $G$ is said to be {\it Haagerup} if there exists an affine isometric action of $G$ on a Hilbert space such that the orbit map of $0$ is proper.
In particular, say that a map $f\colon G\rightarrow L_p(\Omega,\mu)$ is {\it $G$-equivariant}, if there is an affine isometric action $\alpha$ of $G$ on $L_p(\Omega,\mu)$ such that $\forall g,h\in G: f(gh)=\alpha(g)(f(h))$.
If $G$ is finitely generated and equipped with word length relative to a finite generating set, then the Haagerup property is equivalent to the existence of a $G$-equivariant coarse embedding of $G$ into a Hilbert space.
\end{df}
The Haagerup property is a subject of intense study (see \cite{Cherixetal}) and is related to deep conjectures: for example, it is known to imply the Baum-Connes conjecture and associated Novikov conjecture \cite{HK97}, \cite{Tu99}.

In $2004$, Guentner and Kaminker introduce two  numerical invariants to quantify coarse embeddability and the Haagerup property \cite{guekam}.
\begin{df}
Fix $p\geq 1$. Given a finitely generated group $G$ and a measure space $(\Omega,\mu)$, the {\it $L_p$-compression} $R(f)$ of a coarse embedding $f\colon G\rightarrow L_p(\Omega,\mu)$ is defined as the supremum of $r\in [0,1]$ such that
\[ \exists C,D>0, \forall g,h\in G: \frac{1}{C} d(g,h)^r - D \leq \| f(g)-f(h)\| \leq C d(g,h) .\]
The {\it equivariant $L_p$-compression} $\alpha_p^*(G)$ of $G$ is defined as the supremum of $R(f)$ taken over all $G$-equivariant coarse embeddings of $G$ into all possible $L_p$-spaces. Taking the supremum of $R(f)$ over all, also non-equivariant, coarse embeddings, leads to the (non-equivariant) {\it $L_p$-compression} $\alpha_p(G)$ of $G$. If there is no ($G$-equivariant) coarse embedding of the group into an $L_p$-space, then we define the (equivariant) compression to be $0$. One can show that the above definitions do not depend on the chosen finite generating subset.
\end{df}
Equivariant and non-equivariant compression are related to interesting group theoretic properties. 
Indeed, based on a remark by M. Gromov, it was shown that the equivariant and non-equivariant $L_2$-compression are equal for amenable groups (see \cite{CTV}). Moreover, if the equivariant $L_p$-compression of a finitely generated group is greater than $\max(\frac{1}{2},\frac{1}{p})$, then the group is amenable \cite{naoper}. This provides a partial converse to the statement that amenable groups satisfy the Haagerup property. In the non-equivariant setting, we have the result that finitely generated groups with non-equivariant $L_2$-compression greater than $1/2$ satisfy non-equivariant amenability, i.e. property $A$ (which is equivalent to exactness of the reduced $C^*$-algebra) \cite{guekam}.

\subsection{Graph products}
Let $\Gamma$ be a simplicial graph (i.e. without loops or multiple edges). We will use $V\Gamma$ 
and $E\Gamma$ to denote the set of vertices and the set of edges of $\Gamma$ respectively.
An edge can be considered as a $2$-element subset of $V\Gamma$.

\begin{df}\label{def:gp}
Let $\ga$ be a finite simplicial graph. Suppose that $\mathfrak{G}=\{G_v \mid v\in V\ga\}$ is a collection of groups (called \textit{vertex groups}).
The \emph{graph product} $\ga \mathfrak{G}$, of this collection of groups with respect to $\ga$,
is the group obtained from the free product of the $G_v$, $v \in V\ga$, by adding the relations
$$[g_v, g_u]=1  \text{ for all }  g_v\in G_v,\, g_u\in G_u \text{ such that $\{v,u\}$ is an edge of } \ga.$$
\end{df}
Here $[a,b]=a^{-1}b^{-1}ab$ denotes the commutator of $a$ and $b$.

When $\ga$ has no edges, $\GG$ is a free product, 
and when  $\ga$ is a complete graph, 
$\GG$ is a direct product of the groups $G_v$, $v \in V\ga$. 
Graph products were first introduced and studied by E. Green 
in her Ph.D. thesis \cite{Green}. 

When all vertex groups are infinite cyclic, 
$\GG$ is a \textit{right-angled Artin group} 
(also known as graph group or pc-group) 
and when all vertex groups are cyclic of order two, 
$\GG$ is a \textit{right-angled Coxeter groups}.

The structure of graph products can be described in terms of simpler graph products using amalgams (see Lemma \ref{lem:A*_CB}) or semi-direct products (see Lemma \ref{lemma:splitsequence}). Dadarlat and Guentner showed that the class of finitely generated groups embedding coarsely into a Hilbert space is closed under direct products and amalgamated free products \cite{guedad}. This implies that this class is also closed under graph products. On the other hand, the group 
$$\mathrm{SL}_2(\mathbb{Z})\ltimes \mathbb{Z}^2= (\mathbb{Z}/6\mathbb{Z} \ltimes \mathbb{Z}^2)*_{(\mathbb{Z}/2\mathbb{Z}\ltimes \mathbb{Z}^2)} (\mathbb{Z}/4\mathbb{Z} \ltimes \mathbb{Z}^2)$$
is not Haagerup since it has relative property (T) with respect to the normal subgroup (see \cite{Cherixetal}), so 
 the class of groups with the Haagerup property is {\em not} closed under amalgamated free products over infinite groups and under general group extensions. So in the light of Lemmas \ref{lem:A*_CB} and \ref{lemma:splitsequence}, it can be can be considered to be surprising that the Haagerup property is invariant under graph products. Our proof makes crucial use of normal forms of
graphs products, which, in some sense, reveal the CAT(0) cubical geometry behind these groups.

In this respect, we mention another interesting result due to de Cornulier, Stalder and Valette \cite{CSV}. They show that the Haagerup property is stable under wreath products although it is not stable under general group extensions.

\subsection{Outline and main results.}
It was proven by Niblo and Reeves \cite{NibloReeves} 
that groups acting properly on  CAT(0) cube complexes 
have the Haagerup property. 
In particular, right-angled Artin groups have the Haagerup property. 

In Section \ref{s:equivemb}, we construct a proper affine isometric action on a Hilbert space of a graph product of groups with the Haagerup property. Our main Theorem \ref{thm:main} below gives 
a new proof for right-angled Artin groups to be Haagerup 
which is independent from the one of \cite{NibloReeves}.

\noindent{\bf Theorem \ref{thm:main}.} {\it Let $p\geq 1$,  $\ga$ be a simplicial countable graph 
and $\mathfrak{G}=\{G_v\mid v\in V\Gamma\}$ a collection of discrete groups. The graph product 
$G=\ga \mathfrak{G}$ admits a proper affine isometric action on an $L_p$-space if and only if all the 
groups $G_v$ admit proper affine isometric actions on $L_p$-spaces.}

In particular a graph product of groups having the Haagerup property has the Haagerup property.

Next, in Section \ref{s:equicomp}, we study the equivariant $L_p$-compression 
of graph products in terms of the equivariant compression of the vertex groups.
Our main result of this section (Theorem 
\ref{thm:maineqcom}) implies the following corollary.
\medskip

\noindent{\bf Corollary \ref{cor:eqsimplecompression}.} {\it
Let $G=\GG$ be a graph product containing a quasi-isometrically embedded non-abelian free group. The equivariant $L_p$-compression of $G$ satisfies   
$$\min_{v\in V\ga}\left(\frac{1}{p}, \alpha_p^*(G_v)  \right)\leq \alpha^*_p(G)\leq \min_{v\in V\ga}\left(\max \left(\frac{1}{2},\frac{1}{p}\right), \alpha_p^*(G_v)\right).$$
}

In Section \ref{s:nequicom}, we  drop the condition that the embedding of $G$ in an $L_p$-space 
is $G$-equivariant and we study the compression for any coarse embedding of $G$ into 
an $L_p$-space. We describe possible compression functions in Theorem \ref{thm:coarseemb}.
As an immediate corollary, we obtain the value of the non-equivariant $L_p$-compression, 
for $G=\GG$ in terms of the non-equivariant $L_p$-compression of the vertex groups. Precisely,
\medskip

\noindent{\bf Corollary \ref{cor:neqcomp}.} {\it Let $\ga$ be a finite graph, $\G$ a collection of finitely generated groups indexed by $V\ga$ and $G=\GG$ the corresponding graph product. 
For each $p\geq 1$, we have 
\[ \alpha_p(G)=\min\{\alpha_p(G_v)\mid v\in V\Gamma\} .\]
}

Dranishnikov \cite{Dranishnikov} proved that right-angled Coxeter groups
have finite asymptotic dimension and gives an upper bound.  
In Section \ref{s:adim}, using the techniques of \cite{Dranishnikov},
 we show the following.
\medskip

\noindent{\bf Theorem \ref{thm:adim}.} {\it 
Let $\ga$ be a finite simplicial graph and let $\G$ be a family of finitely generated groups indexed by vertices of  $V\ga$. Let $G=\GG$. 
Let $\cC$ be the collection of subsets of $V\ga$ spanning a complete graph.
Then
\begin{equation*}
\adim G \leq \max_{C\in \cC}\sum_{v\in C}\max(1,\adim G_v). \end{equation*}
}


\section{Graph products}\label{s:graphproduct}

Let $\Gamma=(V\ga,E\ga)$ be a simplicial graph. 
For any subset $A\subseteq V\Gamma$, we will denote by $\Gamma_A$ 
 the \textit{full subgraph} of $\Gamma$  with vertex set $A.$ That is,
$V\Gamma_A=A$ and for $a_1,a_2\in A$, $\{a_1,a_2\}\in E\Gamma_A$ if and only if $\{a_1,a_2\}\in E\Gamma$.

The \emph{link} of a vertex $v \in V\Gamma$, denoted $\link_{\Gamma}(v)$, is the subset of vertices adjacent to $v$ (excluding $v$ itself); in other words,
$\link_{\Gamma}(v)\coloneqq \{u\in V\Gamma \mid \{v,u\}\in E\Gamma \}.$ 
The \emph{star} of a vertex $v \in V\Gamma$ , denoted $\st_{\Gamma}(v)$, is the subset of vertices adjacent to $v$ including $v$ itself. Therefore  $\st_{\ga}(v)=\link_{\ga}(v)\cup \{v\}$.

Let $\G=\{G_v\mid v\in V\ga\}$ be a collection of groups and let $G=\GG$ be 
the graph product of $\G$ with respect to $\ga$. It is clear from Definition \ref{def:gp} that the vertex groups are isomorphically embedded in the graph product.

Also, Definition \ref{def:gp} implies that the set $\sqcup_{v\in V\ga} G_v$ is a generating set for $G$.
Thus, any element $g\in G$  may be represented as a \emph{word} $W \equiv (g_1,g_2,\dots, g_n)$
where each $g_i$, called a \emph{syllable} of $W$, 
is an element of some $G_v$ and $g=g_1g_2\dots g_n$.
Let $W \equiv (g_1,g_2,\dots, g_n)$ be a word and suppose that for every $1\leq i \leq n$, $g_i\in G_{v_i}$. The word $W$ is {\it reduced} if 
\begin{equation}\label{eq:reduced}
\mbox{ for  every }1 \le i < j \le n,\mbox{if }v_i=v_j \mbox{ then } g_{i+1}\dots g_{j-1}\notin G_{\st(v_i)}\coloneqq\gen{\cup_{u\in \st(v_i)} G_u}.
\end{equation}

The normal form theorem
(\cite[Thm. 3.9]{Green} or \cite[Thm 2.5]{HsuWise}) states that every element
$g\in \Gamma \mathfrak{G} $ can be represented by a reduced word and a reduced word represents the identity 
if and only if it is the empty word.

Let $g\in G$ and $W\equiv(g_1,\dots,g_n)$ be a reduced word representing $g.$ 
We define the \emph{syllable length}  of $g$ in $G$ to be $|g|_\Gamma=n$. This is well defined and independent of the chosen reduced word representing $g$.

For any subset $A\subseteq V\Gamma,$ the subgroup $G_A\leqslant G$, 
generated by $\{G_v \mid v\in A\}$, is called \textit{special}; according to a standard convention, $G_\emptyset=\{1\}$. 
Every $G_A$ is the graph product
of the groups $\{G_v \mid v \in A\}$ with respect to the graph $\Gamma_A$. 
It is also easy to see that there is a \textit{canonical retraction} 
$\rho_A\colon G \to G_A$ defined (on the generators of $G$) by
$\rho_A(g)\coloneqq g$ for each $g \in G_v$ with $v \in A$,
and $\rho_A(h)\coloneqq 1$ for each $h \in G_u$ with $u \in V\Gamma \backslash A$.

\section{Equivariant embeddability of graph products}\label{s:equivemb}

The main result in this section is Theorem \ref{thm:main}, where given $p\geq 1$, a finite graph of groups $\Gamma$ 
and corresponding proper affine isometric actions of the vertex groups on $L_p$-spaces, 
we construct explicitly a proper affine isometric action of $\GG$ on an $L_p$-space. 
The second author gave such a construction in the free product case which was later generalized to amalgamated free products over finite groups by the second author jointly with Thibault Pillon. Independently, the same construction was also given by Micha\l{}  Marcinkowski. 

\begin{df}
Let $\cV$ be a vector space. The {\it affine group} of $\cV$ is the group
$$\mathrm{Aff}(\cV)=\{\phi \colon \cV\to \cV  \mid  \phi(v)= Av+b; \, A\in \mathrm{GL}(\cV), b\in \cV\} .$$

Let $G$ be a discrete group.
An {\it affine  action} of $G$ on a vector space $\cV$ is a group 
homomorphism $(\pi,b) \colon G \to \mathrm{Aff}(\cV),$ where  $(\pi,b) (g) \colon \cV\to \cV$ is given by $v\mapsto \pi(g)(v)+b(g)$. In particular  $\pi\colon G\to \mathrm{GL}(\cV)$
(called the {\it linear} part) is a group homomorphism and $b\colon G\to \cV$ (called the {\it translation} part) is a 1-cocycle for $\pi$, that is
$$b(gh)=\pi(g)b(h)+b(g) \; \forall g,h\in G. $$

Suppose that  $(\cV, \|\cdot \|)$ is a normed vector space. Saying that the action $(\pi,b)$ is isometric boils down to saying that the linear part is an isometry, i.e. $\|\pi(g)(v)\|=\|v\|$ 
for all $v\in \cV$ and $g\in G$. The action  is {\it proper} if  for every $M>0$ and every $v\in \cV$ there is a finite subset $F$ of $G$ such that $\|\pi(g)(v)+b(g)\|>M$ for all $g\in G\backslash F$.
\end{df}

\begin{rem}\label{rem:b-proper}
An affine isometric action $(\pi, b)$ on $(\cV,\|\cdot \|)$ is proper if and only if the map $b\colon G\to \cV$ is proper, that is,  for every $M>0$ there is a finite subset $F$ of $G$ such that $\|b(g)\|>M$ for all $g\in G\backslash F$.
\end{rem}

\begin{prop}\label{prop:main}
Let $\ga=(V,E)$ be a simplicial graph, 
$\mathfrak{G}=\{G_v\mid v\in V\}$ a collection of groups 
together with  a proper affine isometric action $(\pi_v,b_v)$ of $G_v$ on an $L_p$-space. 
Let $G=\GG$. 
There exists an affine isometric action $(\tau,\beta)$ of $G$ on an $L_p$-space satisfying that 
for every $g\in G$ and every reduced word $(g_1,\dots, g_n)$ representing $g$ 
\begin{equation}\label{eq:sumbi}
\|\beta (g)\|^p=\|\beta(g_1\cdots g_n)\|^p=\sum_{i=1}^n \| b_{v_i}(g_i)\|^p.
\end{equation}
where for $i=1,\dots, n$, $v_i\in V\ga$ and $g_i\in G_{v_i}$.
\end{prop}
\begin{proof}
Let $(\pi_v,b_v)\colon G_v\to \mathrm{Aff}(A_v)$ be a proper affine isometric action 
on an $L_p$-space $A_v$. 
Let $A=\oplus_{v\in V}^p A_v$ be the $l^p$-direct sum of the $A_v$, i.e. $\|(a_v)_{v\in V}\|^p=\sum_{v\in V} \|a_v\|^p$ for every $(a_v)_{v\in V} \in A$. Note that each vertex group $G_v$ admits a unitary action $\pi_v$ on $A_v$ and so by defining the action of $G_v$ on $A\ominus A_v$ to be trivial, we can extend this action to an action $\widetilde{\pi_v}$ of $G_v$ on $A$. Given any  word $(g_1,g_2,\ldots ,g_n)$ representing $g$, we then define
\[ \pi(g)\colon A\rightarrow A, a\mapsto \widetilde{\pi_{v_1}}(g_1) \widetilde{\pi_{v_2}}(g_2)\ldots \widetilde{\pi_{v_n}}(g_n) a, \]
where $g_i\in G_{v_i}$. This is clearly a linear isometry of $A$ and defines a group homomorphism  $\pi \colon G\to \mathrm{GL}(A)$.

Let $T=\bigsqcup_{v\in V} G/G_{\st(v)}$. 
We are going to build an affine isometric action on the $L_p$-space $\oplus_{t\in T}^p A$. 
Formally we will consider the isomorphic space
$$\cA=\left\lbrace f\colon  \bigsqcup_{v\in V} G/G_{\st(v)}\to A \quad \vrule\quad  \sum_{x\in \bigsqcup_{v\in V} G/G_{\st(v)}} \|f(x)\|^{p}<\infty\right\rbrace\cong \oplus^p_{t\in T} A.$$
We define a linear action $\tau \colon G\to \mathrm{GL}(\cA)$ by  
$$(\tau(g)f)(t)=\pi(g)f(g^{-1}t),\; \mbox{for }g\in G, f\in \cA, t\in T.$$
An easy calculation shows that this action is also isometric, i.e. $\sum_{t\in T} \|f(t)\|^p=\sum_{t\in T} \|(\tau(g) f)(t)\|^p$.

For $a\in A$ and $t\in T$, let  $a\chi_t$ denote the element of $\cA$ given by 
$a\chi_t(t')=a$ if $t=t'$ and $a\chi_t(t')=0$ if $t\neq t'$. 

We are going to build a $1$-cocycle for $\tau$. To this end, given $v\in V$, denote first $i_v\colon A_v\rightarrow A$ the natural inclusion map. Then, for every $v\in V$, $g_v\in G_v$, 
let $\beta(g_v)\in \cA$ be the function  $i_v(b_v(g_v))\chi_{1\cdot G_{\st(v)}}$,
 that is 
$$\beta(g_v)(t)=
\begin{cases} 
i_v(b_v(g_v)) &\mbox{if }t=1\cdot G_{\st(v)}\\
0 &\mbox{otherwise}.
\end{cases}$$
We extend $\beta$ to a function $\beta\colon G\to \cA$ using the $1$-cocycle relation. We need to check that $\beta$ is well defined. For that, we need to show that
\begin{enumerate}
\item[(WD1)] For every $v\in V$ and $g,h\in G_v$, we have $\beta(gh)=\tau(g)\beta(h)+\beta(g)$.
\item[(WD2)] For every $\{u,v\}\in E$, $g\in G_v$ and $h\in G_u$,
$\beta(gh)=\beta(hg)$.
\end{enumerate}
In order to simplify the exposition, we allow ourselves to abusively write $b_v(g)$ for $i_v(b_v(g))$ when there is no risk for confusion.

(WD1). We have that for all $t\in T$
\begin{align*}
\tau(g) b_v(h)\chi_{1\cdot G_{\st(v)}}(t) 
& =
(\pi(g) b_v(h))\chi_{1\cdot G_{\st(v)}}(g^{-1}t) \\
& =
(\pi_v(g) b_v(h))\chi_{g\cdot G_{\st(v)}}(t) \\
&= 
(\pi_v(g) b_v(h))\chi_{1\cdot G_{\st(v)}}(t)\mbox{ since }g\in G_v\leq G_{\st(v)}.
\end{align*}
Thus $\tau(g)\beta(h)+\beta(g)
= 
(\pi_v(g) b_v(h)+b_v(g))\chi_{1\cdot G_{\st(v)}}$
which is equal to $b_v(gh) \chi_{1\cdot G_{\st(v)}}$ as required.
This completes the proof of (WD1).

(WD2). By the $1$-cocycle relation $\beta(gh)=\tau(g)\beta(h)+\beta(g)$. Notice that for $t\in T$,
\begin{align*} 
\tau(g)\beta(h)(t)
&=
(\pi(g)b_u(h))\chi_{1\cdot G_{\st(u)}}(g^{-1}t)\\
&=
b_u(h)\chi_{1\cdot G_{\st(u)}}(g^{-1}t) \mbox{ since }g\in G_v, u\neq v \\
&=
b_u(h)\chi_{g\cdot G_{\st(u)}}(t)\\
&=
b_u(h)\chi_{1\cdot G_{\st(u)}}(t)\mbox{ since }g\in G_v\leq G_{\st(u)}.\\
\end{align*} 
Thus $\beta(gh)=b_u(h)\chi_{1\cdot G_{\st(u)}}+b_v(g)\chi_{1\cdot G_{\st(v)}}
=
\beta(g)+\beta(h)$.
A similar argument shows that $\beta(hg)=\beta(g)+\beta(h)$. This completes the proof of (WD2).

Let $W\equiv (g_1\cdots g_n)$ be a reduced word over the alphabet $\sqcup_{v\in V}G_v$.
Suppose that for $i=1,\dots, n$, $g_i\in G_{v_i}$. 
Then using the $1$-cocycle relation we get that
\begin{equation}\label{eq:betared}
\beta(g_1\cdots g_n)
=
\sum_{i=1}^n \pi(g_1\dots g_{i-1})b_{v_i}(g_i)\chi_{(g_1\cdots g_{i-1})G_{\st(v_i)}}.
\end{equation}
We claim that for all $1\leq i<j\leq n$,
$
(g_1\cdots g_{i-1}) G_{\st(v_i)}
\neq
(g_1\cdots g_{j-1}) G_{\st(v_j)}$.
Indeed, if $$
(g_1\cdots g_{i-1}) G_{\st(v_i)}
=
(g_1\cdots g_{j-1}) G_{\st(v_j)}$$
then $v_i=v_j$ and $g_{i}\cdots g_{j-1}\in G_{\st(v_i)}$, which contradicts the fact that $W$ is reduced (recall \eqref{eq:reduced}).

We conclude that 
\begin{equation}
\|\beta(g_1\cdots g_n)\|^p=\sum_{i=1}^n \| b_{v_i}(g_i)\|^p,
\end{equation}
as required.
\end{proof}

\begin{lemma}\label{lem:trick}
Let $(\pi,b)$ be a proper affine isometric action of some group $G$ on an $L_p$-space. Let $C>0$ be any real number.
Then there exists a proper affine isometric action $(\pi', b')$ of $G$ on an $L_p$-space satisfying that $\|b'(g)\|^p\geq \|b(g)\|^p+2C^p$ for all $g\in G\backslash \{1\}$. 
\end{lemma}
\begin{proof}
Let $\pi_1$ be the natural isometric action of $G$ on $\ell^p(G)$ by left translation. For $a\in \C$ and $g\in G$, let $a\chi_g$ be the sequence given by $a\chi_g(h)=a$ if $g=h$ and zero otherwise. Set $b_1(g)=C\chi_g-C\chi_1$, this is a $1$-cocycle for $\pi_1$ and $\|b_1(g)\|^p=2C^p$ for all $g\in G\backslash \{1\}$. Let $(\pi',b')=(\pi\oplus\pi_1, b\oplus b_1)$. Clearly $\|b'(g)\|^p=\|b(g)\|^p+2C^{p}$.
\end{proof}

We can now prove the main result of this section.
\begin{thm}\label{thm:main}
Let $p\geq 1$,  $\ga$ be a countable simplicial graph and $\mathfrak{G}=\{G_v\mid v\in V\Gamma\}$ a collection of discrete groups. The graph product $G=\ga \mathfrak{G}$ admits a proper affine isometric action on an $L_p$-space if and only if all the groups $G_v$ admit proper affine isometric actions on $L_p$-spaces.
\end{thm}
\begin{proof}
Since admitting proper affine isometric actions is inherited by subgroups 
the only if part is clear. 
Conversely, we assume that every vertex groups admits a proper affine isometric action $(\pi_{v_i},b_{v_i})$ on an $L_p$-space.

Enumerating the vertices of $\ga$ as $\{v_1,v_2,\ldots \}$, we can use Lemma \ref{lem:trick} to assume that 
\begin{equation}
\label{eq:big}
\|b_{v_i}(g)\|\geq i,
\end{equation}
for any $i$. Fix any $M>0$ and assume $\|\beta(g)\|\leq M$. By Equation \eqref{eq:sumbi}, if  $(g_1,g_2,\ldots ,g_n)$ is a reduced word representing $g$, then since $\|b_v(g)\|\geq 1$ for all $g\in G_v$ and $v\in V$, it follows that $M^p\geq \|\beta(g)\|^p\geq |g|_{\ga}$, and in particular $n\leq M^p$. Moreover,  as $M\geq \|\beta(g)\| \geq \|b_{v(g_i)}(g_i)\|$, we have by Equation \eqref{eq:big} that $g_i\notin G_{v_j}$ for $j>M$, so the $g_i$ lie in finitely many vertex groups. By properness of the $b_v$, there are only finitely many choices for the elements $g_i$. Hence, the set $\{g\in G \mid   \|\beta (g)\| <M \}$ is finite and $\beta$ is a proper $1$-cocycle.
\end{proof}

In particular, setting $p=2$, we obtain the following as an immediate corollary.
\begin{cor}
Let $\ga$ be a countable simplicial graph and $\mathfrak{G}=\{G_v\mid v\in V\Gamma\}$ a collection of discrete groups with the Haagerup property. Then $G=\GG$, equipped with the discrete topology, has the Haagerup property.
\end{cor}


\section{The behaviour of equivariant compression under graph products}
\label{s:equicomp}

In \cite{guekam}, the authors introduce a numerical invariant called equivariant compression to quantify {\em how Haagerup} a finitely generated group really is. This invariant can also be studied for equivariant embedding of groups into Banach spaces in general.

The equivariant $L_2$-space compression contains interesting information about the group: e.g.  $\alpha_2^*(G)>1/2$ implies amenability of $G$ (see \cite{guekam}).

\begin{ex}\label{ex:diedral}
Let $G=\Z$ or $(\Z/2\Z)*(\Z/2\Z)\cong \Z\rtimes \Z_2$. It is easy to see that there is an affine isometric action $(\pi, b)$ of $G$ on an $L_p$ space, namely $\mathbb{R}$, satisfying that $\|b(g)\|= |g|$. Thus, $\alpha_p^{*}(G)=1$.
\end{ex}

In the Introduction, we defined the equivariant compression by considering all $G$-equivariant coarse embeddings of a group $G$ into $L_p$-spaces. Note that when $f$ is a $G$-equivariant coarse embedding associated to $\alpha=(\pi,b)$, then the distance $\|b(g)-f(g)\|=\|b(1)-f(1)\|$, so $R(f)=R(b)$. One could thus equivalently define the equivariant compression as the supremum of $R(b)$ taken over all $1$-cocycles associated to linear isometric actions of $G$ on $L_p$-spaces.

Naor and  Peres computed in \cite[Remark 2.2]{naoper} the exact value and we record it in the next lemma for future use.
\begin{lemma}\label{lem:compF2}
Let $\F_2$ denote the non-abelian free group on two generators.
Then $\alpha_p^{*}(\F_2)=\max(1/2, 1/p)$ for $p\geq 1$.
\end{lemma}

\begin{prop}\label{prop:equicom_irred}
Let $\ga$ be a finite  graph and let 
$\mathfrak{G}=\{G_v\mid v\in V\Gamma\}$ be a family of finitely generated groups.
For each $v\in V\ga$, let $X_v$ be a finite generating set of $G_v$ and
set $X=\cup X_v$, a finite generating set of $G=\GG$.
 
Choose $p\geq 1$ and let $\rho\colon \R^+\rightarrow \R^+$ be a non-decreasing function such that $\rho^p$ is sub-additive, 
i.e.  that $\rho(x+y)^p\leq \rho(x)^p+\rho(y)^p$ for all  $x,y\geq 1.$ Assume moreover that there is a constant $C>0$ such that $\rho(x)\geq C$ for every $x\geq 1$.

If for every $v\in V\ga$ there is an affine isometric action 
$(\pi_v,b_v)$ on an $L_p$-space satisfying
$$ \forall g\in G_v: \ \|b_v(g)\| \geq \rho(|g|_{X_v}), $$
where $|\cdot |_{X_v}$ is the word length distance on $G_v$ with respect to $X_v$, then there is an affine isometric action of $G$ on an $L_p$-space with compression function $\rho$.
\end{prop}
\begin{proof}

Let $g \in G$ and suppose that $W$ is a reduced word over $\cup_{v\in V} G_v$ representing $g$
$$ W \equiv (h_1,h_2,\ldots ,h_n),$$ and that $h_i$ is an element of the group $G_{v_i}$. Notice that  $|g|_X=\sum |h_i|_{X_i}$ where $h_i\in X_i$ and also for $g\in G_v$,  $\lvert g \rvert_X=\lvert g\rvert_{X_v}$.

Since $\|b_v(g)\|>C>0$ for all $g\in G_v, v\in V\ga$, 
the argument of Theorem \ref{thm:main} gives us that the affine isometric action 
$(\tau, \beta)$ of Proposition \ref{prop:main} of $G$ on a $L_p$-space  is proper.
Recall that the $1$-cocycle $\beta$ satisfies \eqref{eq:sumbi} and hence
\begin{eqnarray*}
\|\beta(g)\|& = & \sqrt[p]{\sum_{i=1}^m \|b_{v(i)}(h_i)\|^p} \\
& \geq & \sqrt[p]{\sum_{i=1}^m \rho( |h_i|_X)^p} \\
& \geq & \sqrt[p]{\rho(\sum_{i=1}^m |h_i|_X)^p} \ \mbox{ (subadditivity of $\rho^p$)}\\
& = & \rho(\sum_{i=1}^m |h_i|_X)=\rho(|g|_X),
\end{eqnarray*}
so $\rho$ is a compression function for $G$.
\end{proof}
We obtain the following immediate corollary.
\begin{cor}\label{cor:lowerbound}
Let $\ga$ be a finite graph and consider the graph product $G=\GG$ where $\mathfrak{G}=\{G_v\mid v\in V\Gamma\}$ is a collection of finitely generated groups. Then
$\alpha_p^{*}(G)\geq \alpha\coloneqq \min\{1/p,\alpha_p^{*}(G_v) \mid v\in V\}$.
\end{cor}
\begin{proof}
The case $\alpha=0$ is trivial, so we may assume there is $\epsilon>0$ with $\alpha>\epsilon>0$.
Assume that each $\rho_v$ is a compression function of some $1$-cocycle $b_v$ associated to a linear isometric action $\pi_v$ of $G_v$ on an $L_p$-space. By Lemma \ref{lem:trick}, we can assume that there is a constant $C>0$ such that $\|b_v(g)\|\geq C>0$ for every $v\in V, g\in G_v$. More precisely, we can assume that $\rho_v:r\mapsto \frac{1}{C_v} r^{\alpha-\epsilon}$ (see Lemma 2.1 in \cite{Dreesen} where the role of $b_v$ is played by $\tilde{f}$). Let $\rho:\R\rightarrow \R, r\mapsto \min(\rho_v(r)\mid v\in V)$. As $\alpha\leq 1/p$, $\rho^p$ is sub-additive. So, by Proposition \ref{prop:equicom_irred}, we conclude that $\rho$ is a compression function associated to a proper affine isometric action of $G$ on an $L_p$-space. By definition of compression, we can let $\epsilon$ go to $0$ and conclude that $\alpha^{*}_p(G)\geq \alpha=\min\{1/p,\alpha_p^{*}(G_v) \mid v\in V\}$ as desired.
\end{proof}

Given groups $G_1,G_2$  with equivariant Hilbert space compression 
$\alpha^{*}_2(G_i)=\alpha_i,$  $i=1,2$, it follows from \cite{Dreesen} 
that the equivariant Hilbert space compression  $\alpha_2^*(G_1\ast G_2)$ is equal to
$\min(\alpha_1,\alpha_2,1/2)$ unless $G_1$ and $G_2$ are both cyclic of order $2$ 
(in which case the equivariant compression is $1$, as we have seen in Example \ref{ex:diedral}). The fact that $\alpha_2^*(G_1\ast G_2) \geq \min(\alpha_1,\alpha_2,1/2)$ also follows from our Corollary \ref{cor:lowerbound}. Moreover as $G_1,G_2<G$, it is clear that $\alpha^*_2(G)\leq \min(\alpha_1, \alpha_2)$. In order to prove that $1/2$ is also an upper bound, the author uses that, when $G_1$ and $G_2$ are not both cyclic of order $2$, then $G_1*G_2$ contains a quasi-isometrically embedded copy of the non-abelian free subgroup $\F_2$ (this follows from Britton's lemma and basic Bass-Serre tree, see the beginning of the proof of Theorem 4.9 in \cite{Dreesen} for more info). Hence, $\alpha^*_2(G)\leq \alpha^*_2(\F_2)=1/2$. In the case of graph products, one can easily find quasi-isometrically embedded free groups if one restricts to irreducible graphs. Let us introduce this concept.

\begin{df}
A simplicial graph $\ga$ is {\it reducible}, if there exists a partition of $V\ga$ into two 
non-empty subsets $A$ and $B$ such that $A\subseteq \link_\ga(B).$ That is, $V\ga=A\cup B$,
$A\cap B =\emptyset$, $A\neq \emptyset \neq B$ and for every $u\in A$, $v\in B$, $\{u,v\}\in E\ga$.

If the graph is not reducible then it is called {\it irreducible}.
\end{df}

Note that when $\ga$ is reducible, say $V\Gamma=A\cup B$, then $G=G_A\times G_B$ and hence the equivariant $L_p$-compression of $G$ equals $\min(\alpha^{*}_p(G_A),\alpha^{*}_p(G_B))$ (see \cite{guekam}). In order to calculate the equivariant $L_p$-compression of $G$, we may thus assume, without loss of generality, that $\ga$ is irreducible. Also, if  some of the vertex groups are trivial, we can consider the subgraph $\ga'$ of $\ga$ that omit them and obtain an isomorphic group.  In order to simplify the exposition, but without loss of generality, we will henceforth assume that all vertex groups are non-trivial.

\begin{lemma}
Let $\ga$ be a finite irreducible graph and consider the graph product $G=\GG$ where $\mathfrak{G}=\{G_v\mid v\in V\Gamma\}$ is a collection of non-trivial groups. If $\lvert V\Gamma \rvert \geq 2$, then $G$ contains a quasi-isometrically embedded non-abelian free subgroup unless $G=\Z_2*\Z_2$.
\label{lemma:freesubgroup}
\end{lemma}
\begin{proof}
Suppose first that $|V\ga|=2$. Then, since $\ga$ is irreducible, there is no edge joining the vertices and $G$ is a free product of the vertex groups. If one of both vertex groups is not cyclic of order $2$, then $G$ contains a copy of $\F_2$. 

Suppose now that  $\lvert V\Gamma \rvert \geq 3$. Let  $v\in V\ga$. 
If $\link_\ga(v)$ is empty, then $G= G_v*G_{V\ga \backslash \{v\}}$ and by hypothesis none of the factors is trivial, and $G_{V\ga\backslash \{v\}}$ is not cyclic of order two. Hence, it contains a non-abelian free subgroup.

Assume next that $\link_\ga(v)$ is non-empty. If for all $u\in V\ga\backslash \link_\ga(v)$, $\link_\ga(v)\subseteq \link_\ga(u)$, then $G=G_{\link_\ga(v)}\times G_{V\ga \backslash \link_\ga(v)}$ contradicting that $\ga$ is irreducible. Thus,  there is $u\in V\ga \backslash \link_\ga(v)$ and $w\in \link_\ga(v)$ such that $w\notin \link_\ga(u)$. Hence $G_{\{u,v,w\}}=(G_v\times G_w)*G_u$ is quasi-isometrically embedded in $G$. As $G_v\times G_w$ has more than two elements, we have that $G$ contains a quasi-isometric copy of the non-abelian free group $\F_2$.
\end{proof}
We summarize the results of this section in the next theorem which deals with graph products over any finite graph.
\begin{thm}\label{thm:maineqcom}
Let $p\geq 1$,  $\ga$ be a finite simplicial graph and $\G=\{G_v\mid v\in V\}$ a collection of non-trivial, finitely generated groups. Let $A_1,\dots, A_n\subseteq V\ga$, such that $\ga_{A_i}$ is a non-empty irreducible subgraph of $\ga$ and $G=G_{A_1}\times \dots \times G_{A_n}$.
Then $\alpha_p^{*}(G)=\min \{\alpha_p^{*}(G_{A_i})\mid i=1,\dots,n\}$ where
\begin{itemize}
\item[(i)] $\alpha_p^{*}(G_{A_i})=\alpha_p^{*}(G_v)$ if $A_i=\{v\}$, 
\item[(ii)] $\alpha_p^{*}(G_A)=1$ if $G_{A_i}$ is infinite dihedral and,
\item[(iii)] in all other cases we have that  $$\min(1/p, \alpha_i) \leq \alpha_p^{*}(G_{A_i}) \leq \min (\alpha_i, \max(1/2,1/p))$$ where $\alpha_i=\min\{\alpha_p^{*}(G_v) \mid v\in A_i\}$. 
\end{itemize}
\end{thm}
\begin{proof}
The formula for $\alpha_p^{*}(G)$ follows from the fact that the direct sum of proper affine isometric actions of the $G_{A_i}$'s gives a proper affine isometric action of $G_{A_1}\times\dots\times G_{A_n}$ \cite{guekam}. The formulas for (i) and (ii) are clear. The formula for (iii)  follows from Corollary \ref{cor:lowerbound} and from Lemma \ref{lemma:freesubgroup} in combination with  $\alpha^*_p(\F_2)=\max(1/2,1/p)$ as we recorded in Lemma \ref{lem:compF2} and the fact that for any quasi-isometrically embedded subgroup $H$ of a group $K$,  $\alpha_p^{*}(H)\geq \alpha_p^{*}(G)$.
\end{proof}
\begin{rem}
Note that for $1\leq p \leq 2$, the inequalities in (iii) in the above result become equalities. In particular, this applies to the much studied case $p=2$, i.e. we have an explicit formula for the equivariant $L_2$-compression of the graph product in terms of the equivariant $L_2$-compressions of the vertex groups.
\end{rem}
\begin{cor} \label{cor:eqsimplecompression}
Let $G=\GG$ be a graph product containing a quasi-isometrically embedded non-abelian free group. The equivariant $L_p$-compression of $G$ satisfies   
$$\min_{v\in V\ga}\left(\frac{1}{p}, \alpha_p^*(G_v)  \right)\leq \alpha^*_p(G)\leq \min_{v\in V\ga}\left(\max \left(\frac{1}{2},\frac{1}{p}\right), \alpha_p^*(G_v)\right).$$
\end{cor}

\section{The behaviour of non-equivariant compression under graph products}
\label{s:nequicom}

So far, we have quantified how fast $1$-cocycles associated to linear isometric actions on $L_p$-spaces go to infinity. Instead of only considering $1$-cocycles, we now consider arbitrary Lipschitz maps of the group into any $L_p$-space ($p\geq 1$). We recall the following definitions from the Introduction.

\begin{df}
Let $(X,\mu)$ be a measure space. The compression of a coarse embedding $f\colon (G,d) \rightarrow L_p(X,\mu)$ is 
the supremum of real numbers $\alpha \in [0,1]$ such that there exists a constant $C>0$ with
$$ \frac{1}{C} d(x,y)^\alpha \leq \|f(x)-f(y)\| ,$$
for all $x,y\in G$. The {\it $L_p$-compression} $\alpha_p(G)$ of $G$ is the supremum of the compressions of all coarse embeddings of $G$ in all possible $L_p$-spaces.
\end{df}

Notice that $\alpha_p$ is invariant under quasi-isometry.

Before proving the main theorem of this section, we need some more facts
about graph products. 
\begin{lemma}\cite[Lemma 3.20]{Green} \label{lem:A*_CB}
Let  $\ga=(V,E)$ is a simplicial graph and $\mathfrak{G}$ a collection of groups indexed by $V$.
For any $v \in V$ the group $G=\GG$ naturally splits as a free amalgamated product: $G=G_A*_{G_C}G_B$, where $C=\link_\ga(v)$, $B=\{v\}\cup\link_\ga(v)$ and $A=V\backslash \{v\}$.
\end{lemma}

The next lemma is well-known and uses standard Bass-Serre theory, (see \cite{Dicks-Dunwoody} or \cite{Serre}). 
 We give a proof for completeness.
\begin{lemma}
\label{lemma:splitsequence}
For each $u\in V\ga$ let $X_u$ be a generating set for $G_u$. Let $X=\cup_{u\in V\ga} X_u$.
For each $v\in V\ga$, let $A_v=V\ga\backslash \{v\}$, $C_v=\link_\ga(v)$ and $T_v$ be a transversal for the right multiplication action of $G_{C_v}$ on $G_{A_v}$.

\begin{enumerate}
\item[(1)]
 The following sequence
$$1\to \ast_{g\in T_v} gG_vg^{-1} \to G \stackrel{\rho_{A_v}}{\longrightarrow} G_{A_v} \to 1 $$
is exact and splits.

\item[(2)] 
For each $v\in V\ga$, let $\phi_v\colon G\to W_v\rtimes G_{A_v}$ be the natural isomorphism corresponding to the split extension, where $W_v=\ker \rho_{A_v}$. For $z\in G$ we write $\phi_v(z)=(w_z^v, g_z^v)$.
Let $l_{W_v}$ denote the word length in $W_v$ with respect to $\cup_{g\in T_v}gX_vg^{-1}$.
Then  \begin{equation}
\label{eq:rewritingsemi}
\sum_{u\in V\ga} l_{W_u}(w_z^u)= l_X(z), \, \forall z\in G.
\end{equation}
\end{enumerate}
\end{lemma}
For the sake of simplicity, when we index free products or sets with $g\in G_A/G_C$, we mean that $g$ runs through some transversal for $G_A/G_C$.  
\begin{proof}
(1). For simplicity we just write $A$ and $C$ instead of $A_v$ and $C_v$.
Let $B= \{v\}\cup C\subset V\ga$. By Lemma \ref{lem:A*_CB},  $G=G_A*_{G_C}G_B$. 
Let $X$ be the graph with vertices $\{{\bf p},{\bf q}\}$ 
and edge ${\bf e}$ connecting ${\bf p}$ to ${\bf q}$.
We form a graph of groups $(X,G(-))$ by setting $G( {\bf p} )=G_A$, 
$G({\bf q} )=G_B$ and $G({\bf e} )=G_C$.
The fundamental group of this graph of groups is $G$  and 
let $T$ be the corresponding Bass-Serre tree. 
In particular, $T$ has vertices  $\{g{\bf p}\}_{g\in G/G_A}\sqcup \{g{\bf q}\}_{g\in G/G_B}$ 
and  edges $\{g {\bf e}\}_{g\in G/G_C}$. 
The edge $g{\bf e}$ connects $h{\bf p}$ and $k{\bf q}$ 
if $gG_C\subseteq hG_A$ and $gG_C\subseteq kG_B$.

Let $K$ denote the kernel of $\rho_A$. 
Notice that the stabiliser of vertex of $T$ of the form $g{\bf p}$ is $gG_Ag^{-1}$.
Since $\rho_A|_{gG_Ag^{-1}}$ is injective,  
$K$ acts freely on the vertices $\{g {\bf p}\}_{g\in G/G_A}$ and 
on the edges $\{g {\bf e}\}_{g\in G/G_C}$ of $T$.
For the vertices of the form $g{\bf q}$, the stabiliser is $gG_Bg^{-1}$.  
Since $G_B=G_v\times G_C$  and $\rho_A$ is injective on $G_C$, 
it follows that $G_B\cap K=G_v$.
Since $K$ is normal, it follows that $K\cap gG_Bg^{-1}=gG_vg^{-1}$.

The graph $K\backslash T$ has two types of vertices. 
For one side, we have vertices of  $\{Kg {\bf p}\}_{g\in K\backslash G /G_A}$. 
Since $K\backslash G /G_A \cong G_A/ G_A$  there is only one of this form.
For the other side, there are vertices of the form 
$\{Kg {\bf q}\}_{g\in K \backslash G/ G_B}$. 
Since $G_B=G_v\times G_C$, and $G_v\leqslant K$, we have that  
$K \backslash G /G_B=  G/K G_C \cong  G_A / G_C$. 
So we have $G_A/G_C$ vertices of the form $Kg{\bf q}$. 
It is easy to check that the set of  edges  $\{Kg {\bf e}\}_{g \in K\backslash G /G_C}$ 
of $K\backslash T$ is in bijection with $G_A/G_C$ and an edge $Kg{\bf e}$ 
connects the vertex $Kg {\bf q}$, $g\in G_A/G_C$ with the unique vertex $Kg {\bf p}$. 
By Bass-Serre theory, $K$ is a free product $*_{g\in G_A/ G_C} gG_vg^{-1}.$

Clearly, the sequence is split, since $\rho_A$ is a retraction.

(2).
It is easy to describe the isomorphism between $G$ and $W_v\rtimes G_{A_v}$ as follows: extend the map 
\[ G_v\cup G_A \rightarrow W_v\rtimes G_{A_v}, \begin{array}{llll} g_v &\mapsto & (g_v,1) & \mbox{ for } g_v\in G_v \\
g_A &\mapsto & (1,g_A) & \mbox{ for } g_A\in G_A \end{array}, \]
defined on the generators of $G$, to the isomorphism of groups $\phi_v$. 

Henceforth, we write $\phi_v(z)=(w_z^v,g_z^v)\in W_V\rtimes G_A$ for $z \in G$.

On $G_{A_v}<G$, we will consider the word length metric $l_{A_v}$ relative to $\cup_{v_i\in V\backslash \{v\}} X_{v_i}$.

Given $z\in G$, write it as a reduced product of elements of the $G_{v_i}$. Group together the elements of $\cup_{v_i\in V\backslash \{v\}} X_{v_i}$ to write 
$g={\bf f_1}g_1^a{\bf f_2}g_2^a \ldots g_n^a$ where the ${\bf f_i}$ are products of elements of $X_v$ and the $g_i^a$ lie in $G_A$. Then 
\begin{equation}
\label{eq:product}
w_z^v={\bf f_1}(g_1^a{\bf f_2}(g_1^a)^{-1})((g_1^ag_2^a){\bf f_3}(g_1^ag_2^a)^{-1}) \ldots (g_1^a\ldots g_n^a){\bf f_n}(g_1^a\ldots g_n^a)^{-1}.
\end{equation}
None of the $g_i^a$ lie in $G_C$ (where $C$ is the link of $v$), because we started with a {\em reduced} product of elements of the $G_{v_i}$. Hence, any two consecutive factors $(g_1^a\ldots g_j^a){\bf f_j}(g_1^a\ldots g_j^a)^{-1}$ and 
$(g_1^a\ldots g_{j+1}^a){\bf f_{j+1}}(g_1^a\ldots g_{j+1}^a)^{-1}$
in the product of Equation \eqref{eq:product}, lie in factor groups of $W_v=\ast_{g\in T} gG_vg^{-1}$ that correspond to distinct cosets $g_1^a\ldots g_{j}^aG_C\neq g_1^a\ldots g_{j+1}^aG_C$.
In particular, $l_{W_v}(w_z^v)$ is equal to the number of $G_v$-letters appearing in a reduced word over $X$ representing $z$. Thus  $$l_G(z)=\sum_{v\in V\ga} l_{W_v}(w_z^v)$$ for all $z\in G$.
\end{proof}

\subsection{The main result}
We recall the Property ($C_p^c$) introduced by Tessera \cite{Tessera} (see also \cite[Definition 2.1]{Hume}).
\begin{df}
A function $f\colon \N\rightarrow \R_{\geq 0}$ is called {\it concave} if $f$ is non-decreasing 
and for all $n\geq m \in \N$, we have
 $$ f(n +m) - f(n) \leq f(n) - f(n -m). $$
We will only consider such functions with $f(0)=0$. So, in order to simplify the proof of our Lemma \ref{lem:concavity} below, we add the condition $f(0)=0$ as part of our definition of concavity.
Let $f\colon \N\rightarrow \R_{\geq 0}$ be a concave function satisfying 
{\it Tessera's property $(C_p)$},
$$ \sum_{n=1}^\infty \frac{1}{n} \left( \dfrac{f(n)}{n} \right)^p <\infty .$$
We say that $f$ satisfies $C_p^c$  if, in addition, $\frac{f(n)^p}{n}$ is 
non-decreasing for all $n$ sufficiently large.
\end{df}
\begin{lemma}
\label{lem:concavity}
Let $f:\N \rightarrow \N$ be a concave function. Then $f(a)+f(b)\geq f(a+b)$ for all $a,b\in \N$.
\end{lemma}
\begin{proof}
Fix any $a,b\in \N$ and assume without loss of generality that $a\leq b$. By concavity, we have
for any $i\in \N$ that
\[ f(a+i)-f(a+i-1)\geq f(a+i+1)-f(a+i).\]
Letting $i$ vary from $0$ to $b-a$ and summing the so obtained $b-a+1$ inequalities, we obtain
\[ f(b)-f(a-1)\geq f(b+1)-f(a),\]
and so
\[ f(b)+f(a)\geq f(b+1)-f(a-1).\]
As this holds for any $a,b\in \N$, we have
\begin{eqnarray*}
f(b)+f(a) & \geq &f(b+1)-f(a-1)\\
& \geq & f(b+2)-f(a-2) \\
&\geq & \ldots \\
&\geq & f(b+a)-f(0)=f(b+a),
\end{eqnarray*}
as desired.
\end{proof}
\begin{ex}
Let $\alpha\in [0,1)$. Notice that $f\colon n\rightarrow n^{\alpha}$ satisfies ($C_p^c$) for $p\geq 1/\alpha$.
\end{ex}

Our main result of this section is the following.
\begin{thm}\label{thm:coarseemb}
Let $\Gamma$ be a finite simplicial graph, 
$\mathfrak{G}=\{G_v\mid v\in V\Gamma\}$ a family of groups each of them generated by a finite set  $X_v$, and $G=\GG$ 
the corresponding graph product. 
Suppose that all the groups $G_v$ are  uniformly coarsely embeddable in an $L_p$-space, i.e.
each group $G_v$ admits a coarse embedding $f_v$ into an $L_p$-space and 
there is a concave map $\rho'\colon\N \rightarrow \N$ with 
$\lim_{t\to \infty} \rho'(t)=+\infty$ such that
\[ \forall v \in V\Gamma, \ \forall x,y\in G_v: \ d_v(x,y)\geq \|f_v(x)-f_v(y)\| \geq \rho'(d_v(x,y)), \]
where $d_v$ is the word length distance on $G_v$ associated to $X_v$.

If $p = 1$, then there is a coarse embedding $\phi$ of $G$ into an $L_p$-space such that
\[ \forall x,y\in G\: \|\phi(x)-\phi(y)\| \geq \rho'(d(x,y)). \]
If $p>1$, then for any function $f\colon \N \rightarrow \N$ satisfying $(C_p^c)$, there is a coarse
embedding $\phi$ of $G$ into an $L_p$-space such that
\[ \forall x,y\in G\: \|\phi(x)-\phi(y)\| \geq \rho(d(x,y)), \]
where $\rho(n)=\min(\rho'(n),f(n))$.
In both cases $d$ is the word length distance on $G$ associated to $X=\cup_{v\in V\ga} X_v$.
\end{thm}

We state an immediate corollary before proceeding to the proof.
\begin{cor}\label{cor:neqcomp}
Let $\ga$ be a finite graph, $\G$ a collection of finitely generated groups indexed by $V\ga$ and $G=\GG$ the corresponding graph product. 
For each $p\geq 1$, we have 
\[ \alpha_p(G)=\min\{\alpha^*_p(G_v)\mid v\in V\Gamma\} .\]
\end{cor}

\begin{proof}[Proof of Theorem \ref{thm:coarseemb}]
Write $V\Gamma=\{v_1,v_2,\ldots ,v_n\}$ and we follow the notation of Lemma \ref{lemma:splitsequence} (2), by setting $G_i=G_{v_i}$, $W_i=W_{v_i}$ and $\mathcal{G}_i=G_{A_{v_i}}$ for $i=1,\dots, n$. 

By $d_{W_i}$ we denote the length distance on $W_i$ induced by $l_{W_i}$.

By Theorem 4.2 in \cite{Hume}, we have the following:
\begin{enumerate}
\item[(i)] If $p = 1$, then there is a coarse embedding $\phi_i'$ of $W_i$ into an $L_p$-space such that
\[ \forall x,y\in W_i: d_{W_i}(x,y) \geq \|\phi_i'(x)-\phi_i'(y)\| \geq \rho(d_{W_i}(x,y)), \]
where $\rho(n)=\rho'(n)$.
\item[(ii)] If $p>1$, then for any function $f\colon \N \rightarrow \N$ satisfying $(C_p^c)$, there is a coarse
embedding $\phi_i'$ of $W_i$ into an $L_p$-space such that
\[ \forall x,y\in W_i: d_{W_i}(x,y)\geq \|\phi_i'(x)-\phi_i'(y)\| \geq \rho(d_{W_i}(x,y)), \]
where $\rho(n)=\min(\rho'(n),f(n))$.
\end{enumerate}
There is a straightforward way of extending $\phi_i'$ 
to a map $\phi_i\colon G=W_i\rtimes \mathcal{G}_i\rightarrow L_p(X,\nu)$. 
Indeed, given $z\in G$ and $i\in \{1,2,\ldots ,n\}$, 
we can write $z=(w_z^i,g_z^i)\in W_i\rtimes \mathcal{G}_i$ and define
\[ \phi_i: (w_z^i,g_z^i)  \mapsto \phi'(w_z^i). \]
Let us first show that $\phi_i$ is Lipschitz.
As $\phi_i'$ is Lipschitz, we know that
\[ \|\phi_i(x)-\phi_i(y) \|  =  \| \phi_i'(w_x^i)-\phi_i'(w_y^i)\| \leq   d_{W_i}(w_x^i,w_y^i)=l_{W_i}((w_x^i)^{-1}w_y^i). \]
On the other hand,
\begin{eqnarray*}
x^{-1}y&=&(w_{x}^i,g_{x}^i)^{-1}(w_y^i,g_y^i) \\
&=& ((g_x^i)^{-1}(w_x^i)^{-1}g_x^i(g_x^i)^{-1}w_y^ig_x^i,(g_x^i)^{-1}g_y^i)\\
&=& ((g_x^i)^{-1}(w_x^i)^{-1}w_y^ig_x^i,(g_x^i)^{-1}g_y^i),
\end{eqnarray*}
so $w_{x^{-1}y}^i=(g_x^i)^{-1}(w_x^i)^{-1}w_y^ig_x^i$ and
\[ l_{W_i}((w_x^i)^{-1}w_y^i)=l_{W_i}((g_x^i)^{-1}(w_x^i)^{-1}w_y^ig_x^i)=l_{W_i}(w_{x^{-1}y}^i) \leq l_G(x^{-1}y).\]
The last inequality follows from Equation \eqref{eq:rewritingsemi}. Define $\phi=(\oplus^p)_{i=1}^n \phi_i$, where $\oplus^p$ denotes the $l_p$-direct sum. Then for all $x,y\in G$, we have
\[ \|\phi(x)-\phi(y)\|_p =(\sum_{i=1}^n \|\phi_i(w_x^i)-\phi_i(w_y^i)\|_p^p)^{1/p} \leq \sum_{i=1}^n \|\phi_i(w_x^i)-\phi_i(w_y^i)\|_p \leq nCd(x,y), \]
so that $\phi$ is also Lipschitz. On the other hand, for $z=x^{-1}y$ we have
\[ \|\phi(x)-\phi(y)\|_p \geq \frac{1}{n} \sum_{i=1}^n \|\phi_i(w_x^i)-\phi_i(w_y^i)\|_p \geq \frac{1}{n} \sum_{i=1}^n \rho(l_{W_i}(w^i_z))\geq \frac{1}{n}\rho(\sum_{i=1}^n l_{W_i}(w^i_z)),
\] by Lemma \ref{lem:concavity}. 

Now, by Lemma \ref{lemma:splitsequence} (2), $l_X(z)=\sum_{i=1}^n l_{W_i}(w^i_z)$ and hence, the result follows.

\end{proof}

\section{A bound for the asymptotic dimension of a graph product}
\label{s:adim}

The asymptotic dimension is an invariant of metric spaces
that was introduced by M. Gromov in \cite{Gromov91} and has been broadly studied since then. Our main reference for this section is \cite{BellDranishnikov}.

\begin{df}\label{def:adim}
A metric space $(M,d)$ has asymptotic dimension $\leq n$ if for every $d<\infty$, there exists a uniformly bounded cover $\mathcal{V}$ of $M$ with $d$-multiplicity $\leq n+1$, i.e. denoting the ball of radius $d$ and centre $x$ by $B_d(x)$, we have
$$\sup_{x\in M} \lvert \{V\in \mathcal{V} \mid V\cap B_d(x)\neq \emptyset \}\rvert \leq n+1.$$

We write $\adim \Gamma\leq n$ if the asymptotic dimension of $\Gamma$ is at most $n,$ and  $\adim \Gamma=n$ if $\adim \Gamma\leq n$ and  $\adim \Gamma\nleq n-1.$
\end{df}
The asymptotic dimension is invariant under coarse equivalence. Hence, we can define the asymptotic dimension of a finitely generable group as the asymptotic dimension of its Cayley graph with respect to a finite generating set.

We will make use of the following facts.

\begin{thm}\cite[Theorem 63, Theorem 82]{BellDranishnikov}. 
For any finitely generated groups $A$ and $B$ 
there is the inequality\begin{equation}\label{eq:dprod}
\adim(A \times B) \leq   \adim A + \adim B.
\end{equation}
If $C$ is a common subgroup of $A$ and $B$, then
\begin{equation}\label{eq:faprod}
\adim(A *_C B) \leq  \max\{ \adim A, \adim B, \adim C + 1\}.
\end{equation}
\end{thm}

Let $\ga$ be a graph and  $\G$ be a family of finitely generated groups indexed by vertices of $V\ga$ and let $G=\GG$. We would like to derive information about the asymptotic dimension of $G$ in terms of the asymptotic dimensions of the vertex groups. By Equations \eqref{eq:dprod} and \eqref{eq:faprod}, it makes sense to consider the subgroups $G_S$ where $S\subseteq V\ga$ spans a complete graph. Clearly, the asymptotic dimension can be bounded from below by the asymptotic dimensions of each subgroup $G_S$. On the other hand, one can ask whether the asymptotic dimension of $G$ equals the maximum of the asymptotic dimensions of the $G_S$ where $S$ is a complete subgraph. This claim is not true: consider the free product of two finite groups and use the fact that the asymptotic dimension of a group is 0 if and only if the group is finite.

As the previous example indicates, the main problem is related to the case where the vertex groups are finite. We remedy this by considering $\max(\adim G_v,1)$ instead of just $\adim G_v$. 
\begin{thm}\label{thm:adim}
Let $\ga$ be a finite simplicial graph and let $\G$ be a family of finitely generated groups indexed by vertices of  $V\ga$. Let $G=\GG$. 
Let $\cC$ be the collection of subsets of $V\ga$ spanning a complete graph.
Then
\begin{equation}\label{eq:adimformula}
\adim G \leq \max_{C\in \cC}\sum_{v\in C}\max(1,\adim G_v). \end{equation}
\end{thm}
We remark that this theorem generalizes the result of Dranishnikov \cite{Dranishnikov}
 who, in our terminology, bounds the asymptotic dimension of a right-angled Coxeter
  group by the size of the   biggest complete subgraph in the defining graph. Notice also that our bound is often sharp, see Example \ref{ex:adim} below. 

\begin{proof}
We argue by induction on $|V\ga|$.
If $V\ga=\{v\}$, then $\adim G=\adim G_v\leq \max(\adim G_v, 1)$ and hence \eqref{eq:adimformula} holds.

Assume that $|V\ga|>1$ and that the theorem holds for graph products over graphs with fewer vertices.
Denote $m=\max_{C\in \cC}\sum_{v\in C}\max(1,\adim G_v)$ and choose some vertex $v\in V\ga$. For any $C\in \cC\cap \link_\ga(v)$, we have that $C\cup \{v\}\in \cC$, so 
\begin{equation}\label{eq:adimlink}
\max_{C\in \cC, C\subseteq \link_\ga(v)}\sum_{v\in C}\max(1,\adim G_v)\leq m -  \max(1,\adim G_v).
\end{equation}

Set $A=\{v\}\cup \link_\ga(v)$, $B=V\ga\backslash \{v\}$ and $C=\link_\ga(v)$.
Since $|B|<|V\ga|$, by induction hypothesis, \eqref{eq:adimformula} holds for $G_B$, 
and since $B\subseteq V\ga$, $\adim G_B \leq m$.
As $|C|<|V\ga|$, formula \eqref{eq:adimlink} implies $\adim G_C \leq m-\max(1,\adim G_v)$.

{\bf Case 1:} $A=\link_\ga (v)\cup \{v\}=V\ga$. 

In this case $G= G_{v}\times G_C$. By \eqref{eq:dprod}, $\adim G\leq \adim G_v+ \adim G_C\leq  \adim G_v +m -\max(1,\adim G_v)\leq m$, as desired.

{\bf Case 2:} $A=\link_\ga(v)\cup\{v\}\neq V\ga$.\\
In this case $|A|<|V\ga|$ and since $A\subseteq V\ga$, using the induction hypothesis, we conclude that $\adim G_A\leq m$. Now, using \eqref{eq:faprod} we have that 
\begin{eqnarray*}
 \adim(G) &=& \adim(G_A*_{G_C}G_B)\\
 &\leq & \max(\adim(G_A),\adim(G_B),\adim(G_C)+1)\\
&\leq & \max(m,m,m-\max(\adim G_v,1)+1)\leq m ,
\end{eqnarray*}
as desired.
\end{proof}
\begin{ex}\label{ex:adim}
Let $\D_{\infty}$ denote the infinite dihedral group, ie. the free product  $\Z/2*\Z/2\Z$. 
Then $G= \Pi_{i=1}^n \D_{\infty}\times \Pi_{i=1}^m \Z$ 
has a natural structure as a graph product where all the vertex groups are either 
$\Z$ or $\Z/2\Z$. 
The underlying graph is a join of $n$ sets of two elements and $m$ sets of one element.
Then, the biggest clique has size $m+n$, and thus from Theorem \ref{thm:adim}, $\adim G\leq m+n$. The group $G$ is virtually
$\Z^{m+n}$ and hence $\adim G=m+n$. 
This shows the bound predicted by the Theorem \ref{thm:adim} is sharp.
\end{ex}

\medskip

\noindent{\textbf{{Acknowledgments}}}
Part of the research of the paper was conducted when the first author was visiting the Erwin Schr\"odinger International Institute in Vienna.
The first author is supported by the swiss SNF grant: FN 200020-137696/1 and  by the MCI (Spain) through project MTM2011-25955. The second author is a Marie Curie IEF research fellow.


\end{document}